\documentclass{aimsessay}
\usepackage[utf8]{inputenc}
\usepackage[numbers]{natbib}
\usepackage{arabtex}
{\makeatletter\long\gdef\@gobble#1{}}
\usepackage{utf8}
\setcode{utf8}
\usepackage{booktabs} % \toprule, \midrule, \bottomrule
\usepackage{array}    % \newcolumntype
\usepackage{float}
\usepackage[below]{placeins} % use \FloatBarrier in the body
\usepackage{caption} %many figures in one figure (note subfigure and subfig are deprecated)
\usepackage{subcaption} %many figures in one figure (note subfigure and subfig are deprecated)
\usepackage{moreverb}   % \verbatimtabinput
\usepackage{hyperref}
\usepackage{epstopdf}
\usepackage{psfrag}
\usepackage{graphics}% Others may be landscape, longtable, algorithm, algorithmic, etc.
\usepackage{lineno} % This package together with lineno.sty numbers every line. Makes it easy for edditing.

\theoremstyle{plain}
\newtheorem{thm}[subsection]{Theorem}
\theoremstyle{definition} 
\newtheorem{lem}[subsection]{Lemma}
\newtheorem{cor}[subsection]{Corollary}

\newtheorem{pro}[subsection]{Proposition}
\newtheorem{exa}[subsection]{Example}
\newtheorem{defn}[subsection]{Definition}

\numberwithin{equation}{section}
\title{The Gelfand-Tsetlin Realisation of Simple Modules and Monomial Bases}
\author{Central European University\\
Amadou Keita (keita\_amadou@student.ceu.edu)
}

\date{{\small December 2018}}
\begin{document}
%\selectlanguage{english}
\pagestyle{empty}
\maketitle
% All other files are included via \input. 
% To compile in texmaker while viewing any of those
% without having to switch back, use
%   Options > Define Current Document as 'Master Document'
% To not have to close a PDF, remove viewpdf from quickbuild 
% and open the PDF (once) manually: it will auto-refresh or with control-r
% 
%-------------------------------------------------------------------------
% The abstract is the first thing we want to see. No acknowledgements or 
% dedications here. Fetch the abstract from a separate file.
% Please write it in English and in your mother tongue.
% An abstract should be less than half a page, so that both abstracts 
% (that is both languages) fit onto one page.
% We number roman numerals until the main body
\pagenumbering{roman}
%-----------------------------------------------------------------------------
% See the acknowledgement.tex file and follow the instructions there.
%\input{acknowledgement}

% Don't go typing out the contents.
%\tableofcontents
\newpage
% We switch to arabic numerals here where your page count starts
% Essays must be 35 pages long *starting here* and up to and including
% the conclcusion. It does not include the acknowledgements or references.
% 
% Figures may differ between topics, but they are not there
% to fill the pages -- they must add meaning.
% In general most figures should be 0.8 times the width of the page
% (perhaps wider in total when two or three columns of figures)
% See the example in chapter one for defining that. Be *consistent*
% in your presentation of information.
\pagenumbering{arabic}
\pagestyle{myheadings}
%-----------------------------------------------------------------------------
% Each chapter goes in a separate file
% Chapter titles are just examples
% Always have a question
% Note the Case Pattern used at AIMS

% Abstracts are usually written in English, with a version in your
% mother tongue underneath
\chapter*{Abstract} 
\addcontentsline{toc}{chapter}{Abstract}
% Don't change anything above this.
The most famous simple Lie algebra is $sl_n$ (the $n \times n$ matrices with trace equals $0$).
The representation theory for $sl_n$ has been one of the most important research areas for the past hundred years and within their the simple finite-dimensional modules have become very important. They are classified and Gelfand and Tsetlin actually gave an explicit construction of a basis for every simple module. We try to understand their work, extend it by providing theorems and proofs, and construct monomial bases of the simple module.
%%%%%%%%%%%%%%%%%%%%%%
%
%A short, abstracted description of your essay goes here. 
%It should be about 100 words long. But write it last.
%
%An abstract is not a summary of your essay: it's an abstraction of that. 
%It tells the readers why they should be interested in your essay but summarises all
%they need to know if they read no further.
%
%The writing style used in an abstract is like the style used in the rest of your essay: concise, clear and direct. 
%In the rest of the essay, however, you will introduce and use technical terms. In the abstract you should
%avoid them in order to make the result comprehensible to all.
%
%You may like to repeat the abstract in your mother tongue.

% At a unviersity like Stellenbosch you *must* produce an abstract in Afrikaans for your masters.
% At AIMS you are encouraged to repeat the abstract in your mother tongue
% French, Igbo, Mlagasy, etc. just write it using LaTeX's special
% characters.
% Arabic students see the arabic.tex file for an example
% Amharic use openoffice and export from there and import a figure here.
% Where the words do not exist put the English work in italics, or use mathematical symbols.

\chapter{Introduction}\label{1}
Lie algebras are interesting because they form a central mathematical crossroad, which relates to a host of important areas such as group theory, number theory, algebraic geometry, differential geometry, topology, particle physics and strings. Therefore, knowledge of Lie algebras is critical in making significant progress in many aspects of these related fields.

In this project, we work with finite-dimensional modules and hence finite-dimensional representation of $sl_n$. This means for $g\in sl_n$, $\exists$ a matrix $G$ of order $\dim R$ defined in such a way that 
\begin{align*}
g \longrightarrow G \text{ and }  f \longrightarrow F \Rightarrow \lambda g + \mu f \longrightarrow \lambda G + \mu F \text{ and } [g,f]\longrightarrow [G,F].
\end{align*}
We will work with partitions. We choose integers $m_1$, $m_2$, $\cdots$, $m_n$ such that 
\begin{align*}
m_1\geq m_2\geq \cdots \geq m_n
\end{align*}
is satisfied. We see that these partitions are quite important because they appear to be the core in constructing representations. 

Also, we can calculate the possible number of entries of the index set as 
\begin{align}
N=\frac{n(n-1)}{2}.
\end{align}
Let $e_{i,j}$ be a matrix of order $\dim R$ which has $1$ at the intersection of the $i^{\text{th}}$ row and the $j^{\text{th}}$ column and zeros in all other places. If $i=j$, then 
\begin{align*}
e_{i,j}=e_{i,i}-e_{i+1,i+1}.
\end{align*}
Also, let $E_{i,j}$ be the matrix of order
\begin{align}\label{1.0.2}
\prod_{1\leq i\leq j\leq n-1}\frac{m_i-m_{j+1}+i-j+1}{i-j+1}=\dim R,
\end{align} 
 which under our representation corresponds to elements $e_{i,j}\in sl_n$. We see that each matrix forms a linear combination of $e_{i,j}$; that is 
 \begin{align*}
 E_{i,j}=\sum_{i,j}a_{i,j}e_{i,j}
 \end{align*}
for some $a_{i,j}$. Therefore, the set $E_{i,j}$ distinctly defines our representation. We can find all representations by explicitly describing all linear transformations $E_{i,j}$. 

The quest for irreducible representations of special linear algebra $sl_n$ can be reformulated simpler: we need matrices $E_{i,j}$ of order $\dim R$ in Equation \eqref{1.0.2} satisfying the following bracket relations:
\begin{align*}
[E_{i,j},E_{j,l}]&=E_{i,l} \text{ when } i\neq l, \\
[E_{i,j},E_{j,i}]&=E_{i,i}-E_{j,j},\\ 
[E_{i_1,j_1},E_{i_2,j_2}]&=0 \text{ when } j_1\neq i_2 \text{ and }i_1\neq j_2. 
\end{align*}
The system $E_{i,j}$ is required to have no invariant subspaces (that is to be irreducible).

The representation theory of $sl_n$ has a unique nature in choosing a partition. For the classification of simple finite dimensional modules, we set the last choice $m_n=0$ in the partition. This controls differences between subsequent choices in a partition. With a given module, we can set a parametrisation of the partition and then construct all bases vectors. If we construct all bases vectors, we will know a highest basis vector. It is called the highest weight vector. 

If $S\neq \{0\}$ is a submodule of $R$, it has a highest weight vector. This is obvious for the fact that $S$ is finite dimensional and so one of its bases vectors must be highest. Consequently, the presence of a highest weight vector in the submodule $S$ implies $S$ is all of $R$. Hence we have Theorem \ref{3.3.2}. 

In Chapter \ref{2}, we provide the literature review. We discuss preliminaries in Chapter \ref{3} including important definitions in Section \ref{3.1}, the module $R$ in Section \ref{3.2}, and the module structure in Section \ref{3.3}. Then, we discuss weight vectors in $R$ in Section \ref{3.4}. In Chapter \ref{5}, we discuss our results obtained. We prove that $sl_n-$module is simple in Section \ref{5.1} and give monomial basis in Section \ref{5.2}. Finally, we conclude on our findings in Chapter \ref{6}.  
 % Introduction is usually a chapter itself.
\chapter{Literature review}\label{2}
A comprehensive theory of infinitesimal transformations was first given by a Norwegian mathematician, Sophus Lie (1842-1899). This was at the heart of his work, on what are now called Lie groups and their accompanying Lie algebras; and the identification of their role in geometry and especially the theory of differential equations. The properties of an abstract Lie algebra are exactly those definitive of infinitesimal transformations, just as the axioms of group theory embody symmetry. The term "Lie algebra'' was introduced in 1934 by  a German mathematician, Hermann Weyl, for what had until then been known as the algebra of infinitesimal transformations of a Lie group. In the year $1888$, Sophus Lie and Freidrich Engel published their work on Infinitesimal Transformations $1$ \citep{engel1888theorie}; two years later, they published the Infinitesimal Transformations $2$ \citep{lie1890theorie} and in the year $1893$, they finally published the Infinitesimal Transformations $3$ \citep{lie1893theorie}. This series forms a very good foundation in the subject. 

For the past century, mathematicians are trying to understand the works of Sophus Lie and then build on that foundation. Some of his findings have no theorems and proofs but they are explicit constructions that we can work with. In the year $1950$, I. M. Gelfand and M. L. Tsetlin gave an explicit construction of a basis for every simple module. In their work, they gave all the irreducible representations of general linear algebra ($gl_n$) but without theorems \citep{gelfand1950finite}. In the year $2015$, V. Futorny, D. Grantcharov and L. E. Ramirez provided a classification and explicit bases of tableaux of all irreducible generic Gelfand-Tsetlin modules for the Lie algebra $gl_n$ \citep{futorny2015irreducible}. In February $2016$, V. Futorny, D. Grantcharov, and L. E. Ramirez in their paper initiated the systematic study of a large class of non-generic Gelfand-Tsetlin modules - the class of $1-$singular Gelfand-Tsetlin modules. An explicit tableaux realization and the action of $gl_n$ on these modules is provided using a new construction which they call derivative tableaux. Their construction of $1-$singular modules provides a large family of new irreducible Gelfand-Tsetlin modules of $gl_n$, and is a part of the classification of all such irreducible modules for $n = 3$ \citep{futorny2016singular}.

In this thesis, we will show that the Gelfand-Tsetlin constructions given in the year $1950$ \cite{gelfand1950finite} form all the irreducible representations of special linear algebra $sl_n$. We will show that $sl_n-$module is simple and also construct monomial basis from these modules.

 % Chapters might go from 2. problem statement, 
                 % through 3. model, to 4. analysis & results
\chapter{Preliminaries} \label{3}
\section{Definitions}\label{3.1} 

\subsection{Definitions under representation}
\begin{defn}[Representation \citep{carter2005lie}] \hfill\newline
Suppose $L$ is a Lie algebra and let $x,y \in L$. The operation
$$\rho :L\longrightarrow \text{End}(R) $$
$$\rho([x,y])=[\rho (x),\rho (y)]\equiv \rho (x)\rho (y)-\rho (y)\rho (x).$$
is a Lie algebra representation. The vector space $R$ is the representation space. The bracket $[\cdot ,\cdot]$ is bilinear and also an endomorphism. That means
$$[\cdot ,\cdot]:\text{End}(R)\times \text{End}(R) \longrightarrow \text{End}(R).$$
There are many types of representation but each of them is either reducible or irreducible.
\end{defn}
\begin{defn}[Reducible Representation \citep{humphreys1972introduction}] \hfill\newline
A representation 
$$\rho :L\longrightarrow \text{End}(R) $$
is called reducible if $\exists$ a sub vector space $U \subset R$ such that 
$$\forall x\in L, \ \forall u\in U, \ \rho (x)u \in U.$$ So, the representation map $\rho$ restricts  to 
$$\rho \mid_u\cdot L\longrightarrow \text{End}(U).$$
\end{defn}
\begin{defn}[Irreducible Representation \citep{humphreys1972introduction}] \hfill\newline
A representation 
$$\rho :L\longrightarrow \text{End}(R) $$
is called irreducible if $\exists$ no sub vector space $U$ of $R$, $U\neq R$ such that $\forall x\in L$, $\forall u\in U$, $\rho (x)u \in U$. This implies that every subspace $U$ is actually all of $R$. 
\end{defn}
Understanding representations (especially irreducible representation) is necessary for understanding this project work. Below is an example of an irreducible representation.
\begin{exa}
Let $L=sl(2,\mathbb{C})$. The basis of $L$ can be written as 
$$E=\left(\begin{matrix}
0 & 1 \\
0 & 0
\end{matrix}\right),\hspace{.5cm} F=\left(\begin{matrix}
0 & 0 \\
1 & 0
\end{matrix}\right) \text{ and }   H=H_{1,1}-H_{2,2}=\left(\begin{matrix}
1 & 0 \\
0 & -1
\end{matrix}\right).$$
Suppose $x,y\in L.$ A representation 
$$\rho :L\longrightarrow \text{End}(\mathbb{C}^2) $$
$$\rho([x,y])=[\rho (x),\rho (y)]\equiv \rho (x)\rho (y)-\rho (y)\rho (x).$$
Suppose $\rho = id$. This implies that $\rho(x)=x$. Then
$$\rho([x,y])= x\cdot y- y\cdot x.$$
Also
$$\rho (H)=\left(\begin{matrix}
1 & 0 \\
0 & -1
\end{matrix}\right),\hspace{.5cm}
\rho (E)=\left(\begin{matrix}
0 & 1 \\
0 & 0
\end{matrix}\right) \text{ and }
\rho (F)=\left(\begin{matrix}
0 & 0 \\
1 & 0
\end{matrix}\right).$$
The bracket of any two of the above invariants operates as
\begin{eqnarray*}
\rho ([H,E]) &=& 2\rho(E),\\
\rho ([H,F]) &=& 2\rho(F),\\
\rho ([E,F]) &=& \rho(H).
\end{eqnarray*}
We know 
\begin{eqnarray*}
[\rho(H),\rho(E)] &=& 2\rho(E)\\
&=& \rho(2E) \text{ linearity property}\\
&=& \rho ([H,E]).
\end{eqnarray*}
So 
$$\text{End}(\mathbb{C}^2)\supseteq \rho (sl(2,\mathbb{C}))=\left\{\left(\begin{matrix}
\alpha & \beta \\
\gamma & \delta 
\end{matrix}\right)\Bigg| \alpha + \delta = 0\right\}.$$
Suppose $v\in\mathbb{C}^2$ and $v=\lambda_1e_1+\lambda_2e_2$, where $e_1,e_2$ are canonical basis and $\lambda_1,\lambda_2$ are coefficients. Let $\lambda_1\neq 0 \text{ and } \lambda_2=0$. Then
$$F\cdot v=\lambda_1e_2,$$
$$E\cdot(F\cdot v)=\lambda_1e_1.$$
Also, let $\lambda_1= 0 \text{ and } \lambda_2\neq 0$. Then
$$E\cdot v=\lambda_2e_1,$$
$$F\cdot(E\cdot v)=\lambda_2e_2.$$
Therefore, $e_1,e_2\in S \Rightarrow S=\mathbb{C}$. Since $\nexists$ an invariant subspace, $\rho$ is an irreducible representation.
\end{exa}
\subsection{Definitions under weight} 
\begin{defn}[Weight Space \citep{carter2005lie}] \hfill\newline
Let us explore the anatomy of weight, weight space, weight vector and highest weight vector as they are crucial in the work of this project.

A Lie algebra 
$$sl_n = u^-\oplus h \oplus u^+,$$
where $u^-$ is all lower triangular matrices, $h$ all diagonal matrices and $u^+$ is all upper triangular matrices. If $R$ is a finite dimensional $sl_n-$module, then $H\in h$ acts on $R$ such that 
$$R=H_1\cdot \xi_1+\cdots + H_n\cdot \xi_n=\bigoplus_{{\lambda}} R_{\lambda}, $$
where $\lambda$ runs over $H^*$(a dual) and 
$$R_{\lambda}=\left\{r\in R \Bigg| H \cdot \xi = \lambda (H)\xi \hspace{.5cm} \forall H \in h\right\}.$$
The weight spaces $R_{\lambda}$ are infinitely many and different from zero when $R$ is infinite dimensional. $R_{\lambda}$ is called a weight space, $\xi$ a weight vector and we called $\lambda$ a weight of $R$.

A highest weight vector (maximal vectors) in $sl_n-$module is a non-zero weight vector $\beta$ in weight space $R_{\lambda}$ annihilated by the action of all upper triangular matrices ($E$). We will prove in this project that a highest weight vector is indeed maximal and hence a generator. 
\end{defn}
\begin{defn}[Algebra] \hfill\newline
An algebra is a ring $A$ which is a vector space such that
\begin{eqnarray*}
\mathbb{C} &\longrightarrow & A\\
a & \longmapsto & a\cdot 1
\end{eqnarray*}
 is a ring homomorphism.
 \end{defn}
\begin{defn}[Monomial Basis] \hfill\newline
Let $M$ be a simple module and $v$ be a highest weight vector, then $M$ is generated by $v$ through applying iterative lower triangular matrices on $v$. We can view this iterated applying as being a product in some algebra (namely the universal enveloping algebra). We have the fixed basis $F_{i,j}$ and consider monomials in these $F_{i,j}$ only.

A given set $B$ of monomials is called monomial basis of $M$ if 
$$
\{\underline{F}^{b}.v \mid b \in B \}
$$
is a basis of $M$.

In this project, we will prove the existence of monomial basis and give some examples. 
\end{defn}
\section{The module R}\label{3.2}
$R$ is a vector space with bases $\xi$ \citep{gelfand1950finite}. These bases depend on the choice of integer partition
\begin{align*}
m_1, m_2 , \cdots , m_n \text{ with }(m_1\geq m_2 \geq \cdots \geq m_n).
\end{align*}
\begin{align}
\xi = \left(\begin{matrix}
p_{1,i}, & i=1,...,n-1 \\
p_{2,i}, & i=1,...,n-2 \\
\ddots & \ddots \\
p_{j,i},& \left\{_{0\leq i \leq n-j}^{1\leq j\leq n-1}\right.
\end{matrix}\right).
\end{align}
In order to understand this basis vector quite well, let us consider rows $(i-1), i\text{ and } (i+1)$ and entry $p_{i,j}$ in $\xi$. For all $p_{i,j}$, if $j<1$ or $j>n-i$, then $p_{i,j}=\emptyset$. Otherwise, the relations of the three rows and specifically the entry $p_{i,j}$ are 
\begin{align*}
\begin{cases}
p_{i-1,j}\geq p_{i,j} \geq p_{i-1,j+1},\\
p_{i+1,j-1}\geq p_{i,j} \geq p_{i+1,j},\\
p_{0,j}:= m_j.
\end{cases}
\end{align*} 
Below is a pictorial representation of $p_{i,j}$.
\begin{figure}[H]
\centering
\includegraphics[scale=.5]{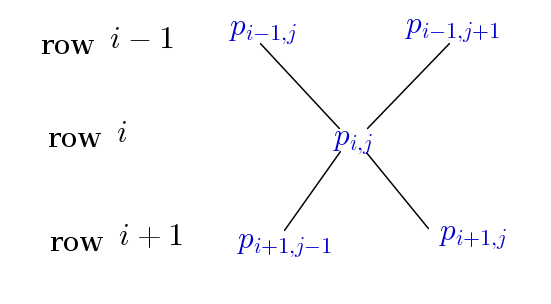}
\caption{Pictogram of entry $p_{i,j}$ in $\xi$.}\label{pic}
\end{figure}
Let $m_1=1$, $m_2=1$ and $m_3=0$. All possible bases from this partition are 
$$\left(\begin{matrix}
1 && 1 \\
 & 1 &
 \end{matrix}\right),\hspace{.2cm}\left(\begin{matrix}
1 && 0 \\
 & 1 &
 \end{matrix}\right),\hspace{.2cm}\left(\begin{matrix}
1 && 0 \\
 & 0 &
 \end{matrix}\right).$$
 
\section{The module structure on R}\label{3.3}
Our representation space $R$ is a $sl_n-$module.
Although this is true, we will not prove it. It is a $sl_n-$module via actions of upper triangular matrices, lower triangular matrices and the diagonal matrices on $\xi$ \citep{gelfand1950finite}. 

For upper triangular matrices in general, as far as the action takes place, the coefficients will never be zero. Let $m_{pq}-p=l_{pq}$ and
\begin{align}
a_{k-1,k}^j=\left[(-1)^{k-1}\frac{\prod_{i=1}^k\left(l_{i,k}-l_{j,k-1}\right)\prod_{i=1}^{k-2}\left(l_{i,k-2}-l_{j,k-1}-1\right)}{\prod_{i\neq j}\left(l_{i,k-1}-l_{j,k-1}\right)\left(l_{i,k-1}-l_{j,k-1}-1\right)}\right]^{\frac{1}{2}}.
\end{align} 
Suppose $\xi_{k-1,k}^j$ is the pattern obtained from $\xi$ by replacing $m_{i,k-1}$ with $m_{i,k-1}+1$. The upper triangular matrix $E_{k-1,k}$ acts on $\xi$ as
\begin{align}\label{3.0.3}
 E_{k-1,k}(\xi)= \sum_{j}a_{k-1,k}^j(\xi_{k-1,k}^j).\end{align}
For a $3\times 3$ matrix, the action of $E_{i,j}$ on $\xi$ raises the $i^{th}$ row in the basis $\xi$ by $1$ on every entry in that row accordingly. The formulas below give explicit descriptions for $n=3$:
\begin{align*}
E_{1,2}\left(\begin{matrix}
p_1 & & p_2 \\
 & q & 
\end{matrix}\right) &= \sqrt{(p_1-q)(q-p_2+1)}\left(\begin{matrix}
p_1 & & p_2 \\
 & q+1 & 
\end{matrix}\right)\\\\
E_{2,3}\left(\begin{matrix}
p_1 & & p_2 \\
 & q & 
\end{matrix}\right) &= \sqrt{\frac{(m_1-p_1)(m_2-p_1-1)(m_3-p_1-2)(p_1-q+1)}{(p_1-p_2+2)(p_1-p_2+1)}}\left(\begin{matrix}
p_1+1 & & p_2 \\
 & q & 
\end{matrix}\right)\\\\
&  + \sqrt{\frac{(m_1-p_2+1)(m_2-p_2)(m_3-p_2-1)(p_2-q)}{(p_1-p_2+1)(p_1-p_2)}}\left(\begin{matrix}
p_1 & & p_2+1 \\
 & q & 
\end{matrix}\right).
\end{align*}
Note that 
\begin{align*}
E_{1,3}=[E_{1,2},E_{2,3}].
\end{align*}
So it is generated by $E_{i,i+1}$. The coefficients as far as action can take place are different from zero. This can be seen in the following cases: we consider coefficients from the actions of $E_{1,2}$ and $E_{2,3}$ in our analysis. We know $p_1\geq q$ and $q\longmapsto q+1$ implying $(p_1-q)\neq 0$. $q\geq p_2$ so $(q-p_2+1)\neq 0$. In the same way, the coefficients from the action of $E_{2,3}$ has $m_1\geq p_1$ and $p_1\longmapsto p_1+1$ implying $m_1>p_1$ so $(m_1-p_1)\neq 0$. Also, $(m_2-p_2)\neq 0$ since $m_2\geq p_2$ and $p_2\longmapsto p_2+1$ implying $m_2>p_2$. Since $m_2>p_2$ and $m_1\geq m_2$, then $(m_1-p_2+1)\neq 0$. Now, $m_1>p_1\geq m_2>p_2\geq m_3$, $\Rightarrow (m_2-p_1-1)\neq 0$. Since $p_1>m_3$, $\Rightarrow (m_3-p_1-2)\neq
 0$. Also, $p_2\geq m_3$, $\Rightarrow (m_3-p_2-1)\neq 0$. For $q\geq p_2$, and $p_2\longmapsto p_2+1$, then $q>p_2$. So $(p_2-q)\neq 0$. We know $p_1\geq q> p_2$, $\Rightarrow (p_1-q+1)\neq 0$. $p_1>p_2$, $(p_1-p_2+2)\neq 0$, $(p_1-p_2+1)\neq 0$ and $(p_1-p_2)\neq 0$. 

The action of $F_{i,j}$ on $\xi$ reduces the entries of the $i^{th}$ row in $\xi$ by $1$ accordingly. This is done in such a way that rules governing the size of entries are conserved. In a general case, let $m_{pq}-p=l_{pq}$ and
\begin{align}
b_{k,k-1}^j=\left[(-1)^{k-1}\frac{\prod_{i=1}^k\left(l_{i,k}-l_{j,k-1}+1\right)\prod_{i=1}^{k-2}\left(l_{i,k-2}-l_{j,k-1}\right)}{\prod_{i\neq j}\left(l_{i,k-1}-l_{j,k-1}+1\right)\left(l_{i,k-1}-l_{j,k-1}\right)}\right]^{\frac{1}{2}}.
\end{align} 
Suppose $\bar{\xi}_{k,k-1}^j$ is the pattern obtained from $\xi$ by replacing $m_{i,k-1}$ with $m_{i,k-1}-1$. The lower triangular matrix $F_{k,k-1}$ acts on $(\xi)$ as
\begin{align}\label{3.0.5}
F_{k,k-1}(\xi)= \sum_{j}b_{k,k-1}^j(\bar{\xi}_{k,k-1}^j).
\end{align}
The formulas for $n=3$ are as follows: 
\begin{eqnarray*}
F_{2,1}\left(\begin{matrix}
p_1 & & p_2 \\
 & q & 
\end{matrix}\right) &=& \sqrt{(p_1-q+1)(q-p_2)}\left(\begin{matrix}
p_1 & & p_2 \\
 & q-1 & 
\end{matrix}\right)\\\\
F_{3,2}\left(\begin{matrix}
p_1 & & p_2 \\
 & q & 
\end{matrix}\right) &=& \sqrt{\frac{(m_1-p_1+1)(m_2-p_1)(m_3-p_1-1)(p_1-q)}{(p_1-p_2+1)(p_1-p_2)}}\left(\begin{matrix}
p_1-1 & & p_2 \\
 & q & 
\end{matrix}\right)\\\\
& & + \sqrt{\frac{(m_1-p_2+2)(m_2-p_2+1)(m_3-p_2)(p_2-q-1)}{(p_1-p_2+2)(p_1-p_2+1)}}\left(\begin{matrix}
p_1 & & p_2-1 \\
 & q & 
\end{matrix}\right).
\end{eqnarray*}
We know that 
\begin{align*}
F_{3,1}=[F_{3,2},F_{2,1}].
\end{align*}
So, $F_{i+1,i}$ generates all $F_{i,j}$ and other actions can be computed using the Lie bracket operation. The coefficients as far as action can take place are different from zero. This is seen in the case that $n=3$. We consider coefficients from the actions of $F_{2,1}$ and $F_{3,2}$ in our analysis. We know $p_1\geq q$, so $(p_1-q+1)\neq 0$. $q\geq p_2$ and $q\longmapsto q-1$, so $(q-p_2)\neq 0$ and $(p_2-q-1)\neq 0$. In the same way, $m_1\geq p_1$ and $p_1\longmapsto p_1-1$ implying $m_1>p_1-1$. So $(m_1-p_1+1)\neq 0$. We know $(m_2-p_1)\neq 0$ since $p_1\geq m_2$ and $p_1-1\geq m_2$ implying $p_1>m_2$. Since $m_1\geq p_1>m_2\geq p_2$, then $m_1> p_2$. So $(m_1-p_2+2)\neq 0$. Now, $m_2\geq p_2\geq m_2$,  $\Rightarrow (m_2-p_2+1)\neq 0$. Since $p_1>p_2\geq m_3$, $\Rightarrow (m_3-p_1-1)\neq
 0$. Also, $p_1\geq q$ and $p_1\longmapsto p_1-1$, $\Rightarrow (p_1-1)\neq q$. So, $p_1-q \neq 0$.  $p_1> p_2$, so $(p_1-p_2+2)\neq 0$, $(p_1-p_2+1)\neq 0$, $(p_1-p_2)\neq 0$. Since $p_2\geq m_3$ and $p_2\longmapsto p_2-1$, then $p_2>m_3$. So $(m_3-p_2)\neq 0$. The coefficients will never be zero when an action is possible.

The diagonal matrices can also be generated by $E_{i,i+1}$ and $F_{i+1,i}$. Some coefficients from the action of $H_{i,i}$ can be zero but not all coefficients. In general 
\begin{align}\label{3.0.6}
H_{i,i}(\xi)=\left(\sum_{i=1}^k m_{i,k}-\sum_{i=1}^{k-1} m_{i,k-1} \right)\left(\xi\right)
\end{align}
where 
$$\left(\sum_{i=1}^k m_{i,k}-\sum_{i=1}^{k-1} m_{i,k-1} \right)$$
is the coefficient of $\xi$.

The action of diagonal matrices ($H_{i,i}$) goes as follows when $n=3$:
\begin{eqnarray*}
H_{1,1}\left(\begin{matrix}
p_1 & & p_2 \\
 & q & 
\end{matrix}\right) &=& q\left(\begin{matrix}
p_1 & & p_2 \\
 & q & 
\end{matrix}\right)\\\\
H_{2,2}\left(\begin{matrix}
p_1 & & p_2 \\
 & q & 
\end{matrix}\right) &=& (p_1+p_2-q)\left(\begin{matrix}
p_1 & & p_2 \\
 & q & 
\end{matrix}\right)\\\\
H_{3,3}\left(\begin{matrix}
p_1 & & p_2 \\
 & q & 
\end{matrix}\right) &=& (m_1+m_2+m_3-p_1-p_2)\left(\begin{matrix}
p_1 & & p_2 \\
 & q & 
\end{matrix}\right).
\end{eqnarray*}

\begin{exa}
Let $m_1=2, m_2=1, m_3=0$. An $sl_n-$module via actions could be observed on all possible bases vectors from this partition. The bases are:
\begin{eqnarray*}
\left(\begin{matrix}
1 && 0 \\
 & 0 &
 \end{matrix}\right),
 \left(\begin{matrix}
1 && 0 \\
 & 1 &
 \end{matrix}\right),
 \left(\begin{matrix}
1 && 1 \\
 & 1 &
 \end{matrix}\right) ,
\left(\begin{matrix}
2 && 0 \\
 & 0 &
 \end{matrix}\right),
 \left(\begin{matrix}
2 && 0 \\
 & 1 &
 \end{matrix}\right),
 \left(\begin{matrix}
2 && 1 \\
 & 1 &
 \end{matrix}\right) ,
\left(\begin{matrix}
2 && 0 \\
 & 2 &
 \end{matrix}\right), 
 \left(\begin{matrix}
2 && 1 \\
 & 2 &
 \end{matrix}\right).
 \end{eqnarray*}
 
 Also, actions of $H_{1,1},H_{2,2},H_{3,3},E_{1,2}, E_{2,3}, F_{2,1}, F_{3,2}$ on these bases vectors are as follows:

\begin{eqnarray*}
H_{1,1}\left(\begin{matrix}
1 && 0 \\
 & 0 &
 \end{matrix}\right) &=&  0\left(\begin{matrix}
1 && 0 \\
 & 0 &
 \end{matrix}\right),\hspace{.5cm}
 H_{1,1}\left(\begin{matrix}
1 && 0 \\
 & 1 &
 \end{matrix}\right) =  1\left(\begin{matrix}
1 && 0 \\
 & 1 &
 \end{matrix}\right),   \\\\
 H_{1,1}\left(\begin{matrix}
1 && 1 \\
 & 1 &
 \end{matrix}\right) &=&  1\left(\begin{matrix}
1 && 1 \\
 & 1 &
 \end{matrix}\right), \hspace{.5cm}
H_{1,1}\left(\begin{matrix}
2 && 0 \\
 & 0 &
 \end{matrix}\right) =  0\left(\begin{matrix}
2 && 0\\
 & 0 &
 \end{matrix}\right),   \\\\
 H_{1,1}\left(\begin{matrix}
2 && 0 \\
 & 1 &
 \end{matrix}\right) &=&  1\left(\begin{matrix}
2 && 0 \\
 & 1 &
 \end{matrix}\right), \hspace{.5cm}
 H_{1,1}\left(\begin{matrix}
2 && 1 \\
 & 1 &
 \end{matrix}\right) =  1\left(\begin{matrix}
2 && 1 \\
 & 1 &
 \end{matrix}\right), \\\\
H_{1,1}\left(\begin{matrix}
2 && 0 \\
 & 2 &
 \end{matrix}\right) &=&  2\left(\begin{matrix}
2 && 0 \\
 & 2 &
 \end{matrix}\right), \hspace{.5cm}
 H_{1,1}\left(\begin{matrix}
2 && 1 \\
 & 2 &
 \end{matrix}\right) =  2\left(\begin{matrix}
2 && 1 \\
 & 2 &
 \end{matrix}\right). 
 \end{eqnarray*}
 \begin{eqnarray*}
H_{2,2}\left(\begin{matrix}
1 && 0 \\
 & 0 &
 \end{matrix}\right) &=&  1\left(\begin{matrix}
1 && 0 \\
 & 0 &
 \end{matrix}\right), \hspace{.5cm}
 H_{2,2}\left(\begin{matrix}
1 && 0 \\
 & 1 &
 \end{matrix}\right) =  0\left(\begin{matrix}
1 && 0 \\
 & 1 &
 \end{matrix}\right),   \\\\
 H_{2,2}\left(\begin{matrix}
1 && 1 \\
 & 1 &
 \end{matrix}\right) &=&  1\left(\begin{matrix}
1 && 1 \\
 & 1 &
 \end{matrix}\right), \hspace{.5cm}
H_{2,2}\left(\begin{matrix}
2 && 0 \\
 & 0 &
 \end{matrix}\right)=  2\left(\begin{matrix}
2 && 0\\
 & 0 &
 \end{matrix}\right),   \\\\
 H_{2,2}\left(\begin{matrix}
2 && 0 \\
 & 1 &
 \end{matrix}\right) &=&  1\left(\begin{matrix}
2 && 0 \\
 & 1 &
 \end{matrix}\right), \hspace{.5cm}
 H_{2,2}\left(\begin{matrix}
2 && 1 \\
 & 1 &
 \end{matrix}\right) = 2\left(\begin{matrix}
2 && 1 \\
 & 1 &
 \end{matrix}\right),  \\\\
H_{2,2}\left(\begin{matrix}
2 && 0 \\
 & 2 &
 \end{matrix}\right) &=&  0\left(\begin{matrix}
2 && 0 \\
 & 2 &
 \end{matrix}\right), \hspace{.5cm}
 H_{2,2}\left(\begin{matrix}
2 && 1 \\
 & 2 &
 \end{matrix}\right) = 1\left(\begin{matrix}
2 && 1 \\
 & 2 &
 \end{matrix}\right).
 \end{eqnarray*}
 \begin{eqnarray*}
H_{3,3}\left(\begin{matrix}
1 && 0 \\
 & 0 &
 \end{matrix}\right) &=&  2\left(\begin{matrix}
1 && 0 \\
 & 0 &
 \end{matrix}\right),\hspace{.5cm}
 H_{3,3}\left(\begin{matrix}
1 && 0 \\
 & 1 &
 \end{matrix}\right) = 2\left(\begin{matrix}
1 && 0 \\
 & 1 &
 \end{matrix}\right),  \\\\
 H_{3,3}\left(\begin{matrix}
1 && 1 \\
 & 1 &
 \end{matrix}\right) &=&  1\left(\begin{matrix}
1 && 1 \\
 & 1 &
 \end{matrix}\right), \hspace{.5cm}
H_{3,3}\left(\begin{matrix}
2 && 0 \\
 & 0 &
 \end{matrix}\right) = 1\left(\begin{matrix}
2 && 0\\
 & 0 &
 \end{matrix}\right),  \\\\
 H_{3,3}\left(\begin{matrix}
2 && 0 \\
 & 1 &
 \end{matrix}\right) &=&  1\left(\begin{matrix}
2 && 0 \\
 & 1 &
 \end{matrix}\right),\hspace{.5cm}
H_{3,3}\left(\begin{matrix}
2 && 1 \\
 & 1 &
 \end{matrix}\right) = 0\left(\begin{matrix}
2 && 1 \\
 & 1 &
 \end{matrix}\right),  \\\\
H_{3,3}\left(\begin{matrix}
2 && 0 \\
 & 2 &
 \end{matrix}\right) &=&  1\left(\begin{matrix}
2 && 0 \\
 & 2 &
 \end{matrix}\right), \hspace{.5cm}
 H_{3,3}\left(\begin{matrix}
2 && 1 \\
 & 2 &
 \end{matrix}\right) =  0\left(\begin{matrix}
2 && 1 \\
 & 2 &
 \end{matrix}\right).
 \end{eqnarray*}
 \begin{eqnarray*}
  E_{1,2}\left(\begin{matrix}
1 && 0 \\
 & 0 &
 \end{matrix}\right) &=&  1\left(\begin{matrix}
1 && 0 \\
 & 1 &
 \end{matrix}\right), \hspace{.5cm}
E_{1,2}\left(\begin{matrix}
2 && 0 \\
 & 0 &
 \end{matrix}\right) =\sqrt{2}\left(\begin{matrix}
2 && 0\\
 & 1 &
 \end{matrix}\right),   \\\\
E_{1,2}\left(\begin{matrix}
2 && 0 \\
 & 1 &
 \end{matrix}\right) &=&  \sqrt{2}\left(\begin{matrix}
2 && 0 \\
 & 2 &
 \end{matrix}\right), \hspace{.5cm}
 E_{1,2}\left(\begin{matrix}
2 && 1 \\
 & 1 &
 \end{matrix}\right) =  1\left(\begin{matrix}
2 && 1 \\
 & 2 &
 \end{matrix}\right).
\end{eqnarray*}
 \begin{eqnarray*}
E_{2,3}\left(\begin{matrix}
1 && 0 \\
 & 0 &
 \end{matrix}\right) &=&  1\left(\begin{matrix}
2 && 0 \\
 & 0 &
 \end{matrix}\right),\hspace{.5cm}
 E_{2,3}\left(\begin{matrix}
1 && 0 \\
 & 1 &
 \end{matrix}\right) = \frac{\sqrt{2}}{2}\left(\begin{matrix}
2 && 0 \\
 & 1 &
 \end{matrix}\right),  \\\\
 E_{2,3}\left(\begin{matrix}
1 && 1 \\
 & 1 &
 \end{matrix}\right) &=&  \frac{\sqrt{3}}{\sqrt{2}}\left(\begin{matrix}
2 && 1 \\
 & 1 &
 \end{matrix}\right),\hspace{.5cm}
 E_{2,3}\left(\begin{matrix}
2 && 0 \\
 & 1 &
 \end{matrix}\right) =  \frac{\sqrt{2}}{2}\left(\begin{matrix}
2 && 1 \\
 & 1 &
 \end{matrix}\right), \\\\
 E_{2,3}\left(\begin{matrix}
2 && 0 \\
 & 2 &
 \end{matrix}\right) &=&  1\left(\begin{matrix}
2 && 1 \\
 & 2 &
 \end{matrix}\right). 
  \end{eqnarray*}
\begin{eqnarray*}
F_{2,1}\left(\begin{matrix}
1 && 0 \\
 & 1 &
 \end{matrix}\right) &=&  1\left(\begin{matrix}
1 && 0 \\
 & 0 &
 \end{matrix}\right), \hspace{.5cm}
 F_{2,1}\left(\begin{matrix}
2 && 0 \\
 & 1 &
 \end{matrix}\right) = \sqrt{2}
  \left(\begin{matrix}
2 && 0 \\
 & 0 &
 \end{matrix}\right),\\\\
F_{2,1}\left(\begin{matrix}
2 && 0 \\
 & 2 &
 \end{matrix}\right) &=&  \sqrt{2}\left(\begin{matrix}
2 && 0 \\
 & 1 &
 \end{matrix}\right), \hspace{.5cm}
 F_{2,1}\left(\begin{matrix}
2 && 1 \\
 & 2 &
 \end{matrix}\right) =  1\left(\begin{matrix}
2 && 1 \\
 & 1 &
 \end{matrix}\right).
 \end{eqnarray*}
 \begin{eqnarray*}
  F_{3,2}\left(\begin{matrix}
1 && 1 \\
 & 1 &
 \end{matrix}\right) &=&  \frac{\sqrt{3}}{\sqrt{2}}\left(\begin{matrix}
1 && 0 \\
 & 1 &
 \end{matrix}\right), \hspace{.5cm} 
F_{3,2}\left(\begin{matrix}
2 && 0 \\
 & 0 &
 \end{matrix}\right) = 1\left(\begin{matrix}
1 && 0\\
 & 0 &
 \end{matrix}\right),   \\\\
 F_{3,2}\left(\begin{matrix}
2 && 0 \\
 & 1 &
 \end{matrix}\right) &=&  \frac{\sqrt{2}}{2}\left(\begin{matrix}
1 && 0 \\
 & 1 &
 \end{matrix}\right),\hspace{.5cm}
 F_{3,2}\left(\begin{matrix}
2 && 1 \\
 & 1 &
 \end{matrix}\right) =  \frac{\sqrt{3}}{\sqrt{2}}\left(\begin{matrix}
1 && 1 \\
 & 1 &
 \end{matrix}\right),   \\\\
 F_{3,2}\left(\begin{matrix}
2 && 1 \\
 & 2 &
 \end{matrix}\right) &=&  1\left(\begin{matrix}
2 && 0 \\
 & 2 &
 \end{matrix}\right).
 \end{eqnarray*}
\end{exa}

The theorem below can now follow suit.
\begin{thm}[$sl_n-$module]\label{3.3.2}
The representation space $R$ is a $sl_n-$module.
\end{thm}
Like we stated from the beginning of this section, we will not prove this theorem. 
\section{Weights and weight vectors in R}\label{3.4}
A highest weight vector is the weight vector that is annihilated by every $(n \times n)$ upper triangular matrix (that is $E_{i,j}$ with $i<j$). We fixed $\xi$ as our basis vector in $R$, the representation space where $q$ is any integer depending on some conditions \citep{gelfand1950finite}. The nature of each basis vector depends on the dimension $n$ of operator $E_{i,j}$ acting on it and the partition. For $n=2$, we choose some integers $m_1$, $m_2$ ($m_1\geq m_2$) such that the condition $m_1 \geq q \geq m_2$ is satisfied. When $n=3$, we choose three integers $m_1$, $m_2$, $m_3$ ($m_1\geq m_2\geq m_3$). The bases vectors in the representation space are now numbered by triples, $p_1$, $p_2$, $q$. The representation is given by 
$$m_1\geq p_1 \geq m_2\geq p_2 \geq m_3 $$
\text{ and }$$ p_1\geq q \geq p_2.$$ 
We have our bases vectors of the form
 $$\xi = \left(\begin{matrix}
p_1 && p_2 \\
&q&
\end{matrix} \right).$$ 
Every weight vector has a corresponding weight. The bases vectors are the weight vectors. Constructing these bases depends on the choices of 
$$m_1\geq m_2 \geq \cdots \geq m_n$$
 as defined above.

Suppose $H$ is a square diagonal matrix. The action
\begin{align}\label{3.4.1}
H_{i,i}(\xi)=\kappa_i (\xi),
\end{align}
 where $\kappa_i$ is the eigenvalue of corresponding weight vector $\xi$. In the relation given by Equation \eqref{3.4.1}, there is a map $\omega$ such that for $h\in H$,
 \begin{eqnarray*}
 \omega :h &\longrightarrow & \mathbb{C} \\ 
 H_{i,i} &\longmapsto & \kappa. 
 \end{eqnarray*}
The map $\omega$ is the weight.
\begin{exa} 
Let $m_1=2$, $m_2=1$ and $m_3=0$. We will construct all weight vectors and then find their eigenvalues ($\kappa_i$) and their corresponding weights. Let $H_{i,i}$, $i=1,2,3$ be the operator. The action of the operator on $\xi$ is as follows:

\begin{eqnarray*}
H_{1,1}\left(\begin{matrix}
1 && 0 \\
 & 0 &
 \end{matrix}\right) &=&  0\left(\begin{matrix}
1 && 0 \\
 & 0 &
 \end{matrix}\right),\hspace{.5cm} \Rightarrow \kappa = 0, \\\\
 H_{1,1}\left(\begin{matrix}
1 && 0 \\
 & 1 &
 \end{matrix}\right) &=&  1\left(\begin{matrix}
1 && 0 \\
 & 1 &
 \end{matrix}\right),\hspace{.5cm} \Rightarrow \kappa =1,  \\\\
 H_{1,1}\left(\begin{matrix}
1 && 1 \\
 & 1 &
 \end{matrix}\right) &=&  1\left(\begin{matrix}
1 && 1 \\
 & 1 &
 \end{matrix}\right),\hspace{.5cm} \Rightarrow \kappa =1,  \\\\
 \end{eqnarray*}
 \begin{eqnarray*}
H_{1,1}\left(\begin{matrix}
2 && 0 \\
 & 0 &
 \end{matrix}\right) &=&  0\left(\begin{matrix}
2 && 0\\
 & 0 &
 \end{matrix}\right),\hspace{.5cm} \Rightarrow \kappa =0,  \\\\
 H_{1,1}\left(\begin{matrix}
2 && 0 \\
 & 1 &
 \end{matrix}\right) &=&  1\left(\begin{matrix}
2 && 0 \\
 & 1 &
 \end{matrix}\right),\hspace{.5cm} \Rightarrow \kappa = 1, \\\\
 H_{1,1}\left(\begin{matrix}
2 && 1 \\
 & 1 &
 \end{matrix}\right) &=&  1\left(\begin{matrix}
2 && 1 \\
 & 1 &
 \end{matrix}\right),\hspace{.5cm} \Rightarrow \kappa = 1, \\\\
H_{1,1}\left(\begin{matrix}
2 && 0 \\
 & 2 &
 \end{matrix}\right) &=&  2\left(\begin{matrix}
2 && 0 \\
 & 2 &
 \end{matrix}\right),\hspace{.5cm} \Rightarrow \kappa = 2,\\\\
 H_{1,1}\left(\begin{matrix}
2 && 1 \\
 & 2 &
 \end{matrix}\right) &=&  2\left(\begin{matrix}
2 && 1 \\
 & 2 &
 \end{matrix}\right),\hspace{.5cm} \Rightarrow \kappa = 2, \\\\
 \end{eqnarray*}
 \begin{eqnarray*}
  H_{2,2}\left(\begin{matrix}
1 && 0 \\
 & 0 &
 \end{matrix}\right) &=&  1\left(\begin{matrix}
1 && 0 \\
 & 0 &
 \end{matrix}\right),\hspace{.5cm} \Rightarrow \kappa = 1, \\\\
 H_{2,2}\left(\begin{matrix}
1 && 0 \\
 & 1 &
 \end{matrix}\right) &=&  0\left(\begin{matrix}
1 && 0 \\
 & 1 &
 \end{matrix}\right),\hspace{.5cm} \Rightarrow \kappa =0,  \\\\
  H_{2,2}\left(\begin{matrix}
1 && 1 \\
 & 1 &
 \end{matrix}\right) &=&  1\left(\begin{matrix}
1 && 1 \\
 & 1 &
 \end{matrix}\right),\hspace{.5cm} \Rightarrow \kappa =1,  \\\\
 H_{2,2}\left(\begin{matrix}
2 && 0 \\
 & 0 &
 \end{matrix}\right) &=&  2\left(\begin{matrix}
2 && 0\\
 & 0 &
 \end{matrix}\right),\hspace{.5cm} \Rightarrow \kappa =2,  \\\\
 H_{2,2}\left(\begin{matrix}
2 && 0 \\
 & 1 &
 \end{matrix}\right) &=&  1\left(\begin{matrix}
2 && 0 \\
 & 1 &
 \end{matrix}\right),\hspace{.5cm} \Rightarrow \kappa = 1, \\\\
 H_{2,2}\left(\begin{matrix}
2 && 1 \\
 & 1 &
 \end{matrix}\right) &=&  2\left(\begin{matrix}
2 && 1 \\
 & 1 &
 \end{matrix}\right),\hspace{.5cm} \Rightarrow \kappa =2  \\\\
H_{2,2}\left(\begin{matrix}
2 && 0 \\
 & 2 &
 \end{matrix}\right) &=&  0\left(\begin{matrix}
2 && 0 \\
 & 2 &
 \end{matrix}\right),\hspace{.5cm} \Rightarrow \kappa =0, \\\\
  H_{2,2}\left(\begin{matrix}
2 && 1 \\
 & 2 &
 \end{matrix}\right) &=&  1\left(\begin{matrix}
2 && 1 \\
 & 2 &
 \end{matrix}\right),\hspace{.5cm} \Rightarrow \kappa = 1,
 \end{eqnarray*}
 \begin{eqnarray*}
H_{3,3}\left(\begin{matrix}
1 && 0 \\
 & 0 &
 \end{matrix}\right) &=&  2\left(\begin{matrix}
1 && 0 \\
 & 0 &
 \end{matrix}\right),\hspace{.5cm} \Rightarrow \kappa =2,  \\\\
 H_{3,3}\left(\begin{matrix}
1 && 0 \\
 & 1 &
 \end{matrix}\right) &=&  2\left(\begin{matrix}
1 && 0 \\
 & 1 &
 \end{matrix}\right), \hspace{.5cm} \Rightarrow \kappa = 2, \\\\
 \end{eqnarray*}
 \begin{eqnarray*}
 H_{3,3}\left(\begin{matrix}
1 && 1 \\
 & 1 &
 \end{matrix}\right) &=&  1\left(\begin{matrix}
1 && 1 \\
 & 1 &
 \end{matrix}\right),\hspace{.5cm} \Rightarrow \kappa =1,  \\\\
H_{3,3}\left(\begin{matrix}
2 && 0 \\
 & 0 &
 \end{matrix}\right) &=&  1\left(\begin{matrix}
2 && 0\\
 & 0 &
 \end{matrix}\right),\hspace{.5cm} \Rightarrow \kappa =1,  \\\\
 H_{3,3}\left(\begin{matrix}
2 && 0 \\
 & 1 &
 \end{matrix}\right) &=&  1\left(\begin{matrix}
2 && 0 \\
 & 1 &
 \end{matrix}\right),\hspace{.5cm} \Rightarrow \kappa = 1, \\\\
 H_{3,3}\left(\begin{matrix}
2 && 1 \\
 & 1 &
 \end{matrix}\right) &=&  0\left(\begin{matrix}
2 && 1 \\
 & 1 &
 \end{matrix}\right),\hspace{.5cm} \Rightarrow \kappa = 0, \\\\
H_{3,3}\left(\begin{matrix}
2 && 0 \\
 & 2 &
 \end{matrix}\right) &=&  1\left(\begin{matrix}
2 && 0 \\
 & 2 &
 \end{matrix}\right),\hspace{.5cm} \Rightarrow \kappa = 1,\\\\
 H_{3,3}\left(\begin{matrix}
2 && 1 \\
 & 2 &
 \end{matrix}\right) &=&  0\left(\begin{matrix}
2 && 1 \\
 & 2 &
 \end{matrix}\right),\hspace{.5cm} \Rightarrow \kappa = 0.
 \end{eqnarray*}
The coefficients of basis vectors above are the eigenvalues and they help us calculate the weights for corresponding weight vector $\xi$. For any weight vector $\xi$, the corresponding weight in the case $n=3$ as defined by the map $\omega$ is as follows:
 \begin{eqnarray*}
\text{wight} \left(\begin{matrix}
1 && 0 \\
 & 0 &
 \end{matrix}\right) & =  \varepsilon_2 + 2\varepsilon_3,  
 \hspace{1cm}\text{wight} \left(\begin{matrix}
1 && 0 \\
 & 1 &
 \end{matrix}\right) &  = \varepsilon_1 +  2\varepsilon_3, \\\\
\text{wight} \left(\begin{matrix}
1 && 1 \\
 & 1 &
 \end{matrix}\right) &= \varepsilon_1 + \varepsilon_2 + \varepsilon_3,  
\hspace{1cm}\text{wight} \left(\begin{matrix}
2 && 0 \\
 & 0 &
 \end{matrix}\right)  & = 2\varepsilon_2 + \varepsilon_3,  \\\\
\text{wight} \left(\begin{matrix}
2 && 0 \\
 & 1 &
 \end{matrix}\right)  &= \varepsilon_1 + \varepsilon_2 + \varepsilon_3, 
\hspace{.3cm}\text{wight} \left(\begin{matrix}
2 && 1 \\
 & 1 &
 \end{matrix}\right) &  = \varepsilon_1 + 2\varepsilon_2 , \\\\
\hspace{1.5cm}\text{wight} \left(\begin{matrix}
2 && 0 \\
 & 2 &
 \end{matrix}\right)  & = 2\varepsilon_1 + \varepsilon_3,
\hspace{1cm}\text{wight} \left(\begin{matrix}
2 && 1 \\
 & 2 &
 \end{matrix}\right)& = 2\varepsilon_1 + \varepsilon_2.
 \end{eqnarray*}
\end{exa}
Now, for arbitrary partition
\begin{align*}
m_1&\geq m_2\geq \cdots \geq m_n,\\
&H_{i,i}(\xi)=\kappa_i (\xi)
\end{align*}
has weight 
\begin{align*}
\omega = \kappa_1\varepsilon_1+\kappa_2\varepsilon_2+\cdots + \kappa_n\varepsilon_n
\end{align*}
where $\varepsilon_1$ is the weight for $\xi_1$, $\varepsilon_1 + \varepsilon_2$ the weight of $\xi_2$ and so on. Since $sl_n$ is trace free, 
\begin{align}
\varepsilon_1+\cdots +\varepsilon_n=0.
\end{align}
By definition, 
\begin{align*}
\omega : h&\longrightarrow  \mathbb{C},\\
H_{i,i}-H_{i+1,i+1}&\longmapsto 1, \\
H_{j,j}-H_{j+1,j+1}&\longmapsto  0, \hspace{1cm} \forall j\neq i. \\
\end{align*}
Then for $H\in h$, 
$$H\left(\begin{matrix}
1 && 0 \\
 & 0 &
 \end{matrix}\right)  = a\left(\begin{matrix}
1 && 0 \\
 & 0 &
 \end{matrix}\right), $$
 where $a$ is an eigenvalue.
So 

\begin{eqnarray*}
H &=& \omega_1(H_{1,1}-H_{2,2})+\omega_2(H_{2,2}-H_{3,3})\\
H \left(\begin{matrix}
1 && 0 \\
 & 0 &
 \end{matrix}\right)&=& \left(\omega_1(H_{1,1}-H_{2,2})+\omega_2(H_{2,2}-H_{3,3})\right)\left(\begin{matrix}
1 && 0 \\
 & 0 &
 \end{matrix}\right)\\\\
&=& \left[\omega_1(\kappa_1-\kappa_2)+\omega_2(\kappa_2-\kappa_3)\right]\left(\begin{matrix}
1 && 0 \\
 & 0 &
 \end{matrix}\right)\\\\
 &=& \left[\omega_1(-1)+\omega_2(-1)\right]\left(\begin{matrix}
1 && 0 \\
 & 0 &
 \end{matrix}\right)\\\\
 &=& \left(-\omega_1 -\omega_2 \right)\left(\begin{matrix}
1 && 0 \\
 & 0 &
 \end{matrix}\right).
\end{eqnarray*}
The weight for $\left(\begin{matrix}
1 && 0 \\
 & 0 &
 \end{matrix}\right)$ 
 is $-\omega_1 - \omega_2$ which is the same as $\varepsilon_1+2\varepsilon_3$ above.

In general, 
$$H_{i,i}(\xi)=\left(\sum_{i=1}^k m_{i,k}-\sum_{i=1}^{k-1} m_{i,k-1} \right)\left(\xi\right)$$
where 
$$\left(\sum_{i=1}^k m_{i,k}-\sum_{i=1}^{k-1} m_{i,k-1} \right)$$
is the coefficient of $\xi$.

We can compute any weight from the above relation. The relation works by summing the corresponding row entries of $k^{\text{th}}$ row and then taking away the sum of all entries of the $(k-1)^{\text{th}}$ row. For any $n$, the highest ($n^{\text{th}}$) row for each basis vector corresponds on the the fixed $m_1$ $m_2$ $\cdots$ $m_n$. 

Suppose we let $m_1=3$, $m_2=2$ and $m_3=0$. Then $$E_{1,2}^{l_3}E_{2,3}^{l_2}E_{1,2}^{l_1}\xi_q=\beta$$
where $\beta$ is a highest weight vector and $l_i$ is the maximum times each operator can act on $\xi$ while all conditions are observed to either raise the first row or the second row of $\xi$. Due to the nature of transitions as a consequence of the the action of the sequence, the result is unique.

The representation space $R$ is simple if for all $v\in R$, $\exists$ upper triangular square matrices such that 
\begin{align*}
E^{\underbar{a}(\alpha)}\cdot v=\beta,
\end{align*}
a highest weight vector. The weight vectors could be of the form 
\begin{align}
v=\sum_{\alpha (\text{patterns})}a_{\alpha} \alpha
\end{align}
 where $\alpha$ is a weight vector. We will show that there exists a sequence $E^{\underbar{ a}(\alpha)}$ such that its action on any sum of weight vectors annihilates all but one. That remaining weight vector is a highest weight vector. For example, let $m_1=4$, $m_2=2$ and $m_3=0$. We need $E_{2,3}^2E_{1,2}$ to act on the sum $(\alpha_1 + \alpha_2)$ to yield exactly the highest weight vector when $$\alpha_1=\left(\begin{matrix}
4 & & 0 \\
 & 3 & 
\end{matrix}\right) \text{ and } \alpha_2=\left(\begin{matrix}
4 & & 2 \\
 & 2 & 
\end{matrix}\right).$$ 
The process becomes most critical when there are many basis vectors to act on. The action yields new basis vectors and annihilates some others in the process. With a sequence that acts maximally on these basis vectors, they will all be annihilated except for one. Any more post actions will annihilate the resulting vector. For $m_1=3$, $m_2=2$, and $m_3=0$, we want to find a sequence $\underbar{E}$ such that  
$$\underbar{E}\left( \left(\begin{matrix}
3 && 1 \\
& 1 & \\
\end{matrix}\right)+\left(\begin{matrix}
3 && 1 \\
& 2 & \\
\end{matrix}\right)+\left(\begin{matrix}
2 && 2 \\
& 2 & \\
\end{matrix}\right)+\left(\begin{matrix}
3 && 0 \\
& 1 & \\
\end{matrix}\right)+\left(\begin{matrix}
2 && 1 \\
& 1 & \\
\end{matrix}\right)\right)$$
gives a highest weight vector. We see that $E_{1,2}^2$ annihilates three basis vectors already and raises two others to 
$$\left( \left(\begin{matrix}
3 && 1 \\
& 3 & \\
\end{matrix}\right)+\left(\begin{matrix}
3 && 0 \\
& 3 & \\
\end{matrix}\right)\right).$$ The action of a $E_{2,3}$ on the resulting basis vectors annihilates one and raises the other to yield a highest weight vector. That is 
$$\left( \begin{matrix}
3 && 2 \\
& 3 & \\
\end{matrix}\right).$$
 The sequence 
 \begin{align*}
 \underbar{E}=E_{2,3}E_{1,2}^2.
 \end{align*}
 If this process is always true, then $R$ is a simple $sl_n-$module. We proved this in Theorem \ref{5.1.1}.

 % You do not need to have exactly 4 chapters.
                 % It is probably a good minimum, with 5 chapters 
                 % average, and 7 chapters might be a maximum.
%\input{chapter4} % Conclusion is usually a chapter itself. 
\chapter{Results and discussion}\label{5}
\section{Simple $sl_n-$module}\label{5.1}
\begin{thm}[Simple $sl_n-$module]\label{5.1.1}
The representation space $R$ is a simple $sl_n-$module.
\end{thm}
This theorem requires a proof for many parts so we break it down into two propositions and two lemmas.
\begin{pro}\label{5.1.2}
For every given partition $\exists$ a highest weight vector, $\beta$.
\end{pro}
\begin{proof}
Suppose for integers $m_1$, $m_2$, $...$, $m_n \text{ with }(m_1\geq   \cdots \geq m_n)$ that
\begin{align*}
\beta =\left(\begin{matrix}
m_1 && m_2 && m_3 &\cdots& m_n \\
&m_1 && m_2 && \cdots& \\
 && \ddots & & \ddots && \\
 && & m_1 & & &
\end{matrix}\right).
\end{align*}
Suppose $\exists$ some $\xi_i$ such that $$\xi_1 < \xi_2 < \cdots <\xi_s$$
and  $E_{i,i+1}\cdot\xi_1 \neq 0$ and $E_{i,i+1}\cdot\xi_2\neq 0$ and so on. Also, suppose that entries in both basis vectors $\xi_1$, $\xi_2$ are equal at the bottom, except for a certain row 
such that in that row, 
\begin{align*}
^1\xi_1+^2\xi_1+\cdots +^s\xi_1 < \hspace{.1cm} ^{a_1}\xi_2+^{a_2}\xi_2+\cdots +^{a_s}\xi_2 
\end{align*}
and that 
\begin{align*}
^1\xi_1<(^2\xi_1+\cdots +^s\xi_1) < ( ^{a_1}\xi_2+^{a_2}\xi_2+\cdots +^{a_s}\xi_2 ).
\end{align*}
We can write 
\begin{align*}
\xi_1=\left(v-\xi_i\right).
\end{align*}
The action 
\begin{align*}
E_{i,i+1}\cdot v = E_{i,i+1}\cdot \xi_1+E_{i,i+1}\cdot(v-\xi_1) = ^1\xi_1+ \sum_{\Psi_{i,j}>^1\xi_1}\Psi_{i,j}.
\end{align*}
Now, we have
 \begin{align*}
 E^{a(\xi_1)}\cdot v = E^{a(\xi_1)}\cdot \sum_{i} c_i\xi_i \text{ for } c_i\neq 0.
 \end{align*}
The sequence is actually raising the weight vectors  by the series of actions and the supposedly the smallest basis vector becomes a highest weight vector as a consequence.
So  \begin{align*}
 E^{a(\xi_1)}\cdot \sum_{i}c_i\xi_i=\lambda_{\beta}\beta +\sum_{\Psi_{i,j}>\beta}\Psi_{i,j}=\lambda_{\beta}\beta \text{ for } \lambda_{\beta}\neq 0.
\end{align*}
Therefore, $\beta$ is a highest weight vector.

The weight for $\beta$ is such that
\begin{eqnarray*}
H_{i,i}\cdot \beta &=& \left(\sum_{k=1}^im_{k,i} - \sum_{k=1}^{i-1}m_{k,i-1}\right)\cdot \beta \\
 &=& [q + (p_1+p_2-q)+(m_1+m_2+m_3-p_1-p_2)+\cdots \\
 && + q+(p_1+p_2-q)+ \cdots +(m_1+m_2+\cdots +m_n-p_1-p_2-\cdots -p_{n-1})]\cdot \beta \\
 &=& (c_1\varepsilon_1 + c_2\varepsilon_2 +c_3\varepsilon_3 + \cdots + c_{n-1}\varepsilon_{n-1} + c_n\varepsilon_n) \cdot \beta.
\end{eqnarray*}
So $\beta$ has weight 
 \begin{align*}
 c_1\varepsilon_1 + c_2\varepsilon_2 + \cdots + c_{n-1}\varepsilon_{n-1} + c_n\varepsilon_n,
 \end{align*}
which is a highest weight.

\end{proof}
\begin{pro}
\label{5.1.3}
 For any basis vector $\xi$, there exists a vector $a(\xi) \in \mathbb{Z}_{\geq 0}^N$, where $N = \frac{n(n-1)}{2}$, of upper triangular matrices such that
\begin{align*}
E^{a(\xi)}.\xi=\lambda_{\beta}\beta \text{ for } \lambda_{\beta}\neq 0,
\end{align*}
where 
\begin{align*}
E^{a(\xi)} = E_{1,2}^{a_N} \left(E_{2,3}^{a_{N-1}}E_{1,2}^{a_{N-2}}\right)\cdots\left(E_{n-1,n}^{a_{n-1}}\cdots E_{2,3}^{a_2}E_{1,2}^{a_1}\right)
\end{align*}
\end{pro}
\begin{proof}
From the order 
 \begin{align*}
 \xi_1 < \xi_2 < \cdots <\xi_s
 \end{align*}
 introduced in Proposition \ref{5.1.2}, we see that $\xi_1$ is smaller than all other basis vectors. The action $$E^{a(\xi)}.\xi=E^{a(\xi_1)}\cdot \xi_1+E^{a(\xi_1)}\cdot\left(\sum_{\Psi_{i,j}>\xi_1}\Psi_{i,j}\right).$$
 But 
 $$\left(\sum_{\Psi_{i,j}>\xi_1}\Psi_{i,j}\right)$$ 
  will be annihilated by the action since its elements are bigger and $E^{a(\xi)}$ will be the sequence that raises $\xi_i$ to $\beta$, which is a highest weight.
\end{proof}
\begin{lem}
\label{5.1.4}
 Suppose $v$ is a non-zero element in $R$, $$v=\sum_{i=1}^sc_i\xi_i, \text{ with } c_i\neq 0 \in \mathbb{C}$$
Then there exists a sequence of upper triangular matrices such that

$$E^{a(v)}\cdot v=\lambda_{\beta}\beta \text{ where } \lambda_{\beta} \neq 0. $$
\end{lem}
\begin{proof}
From Proposition \ref{5.1.2}, for $\xi_1 < \xi_2$, we established that
$$^1\xi_1<(^2\xi_1+\cdots +^s\xi_1) < ( ^{a_1}\xi_2+^{a_2}\xi_2+\cdots +^{a_s}\xi_2 ).$$ 
Then, for all $c_i\neq 0$, 
\begin{align*}
E^{a(\xi_i)}\cdot v & = E^{a(\xi_i)} \cdot \left(\sum_{i=1}^s {c_i} {\xi_i} \right)\\
&= E^{a(\xi_i)}\cdot \left(c_1 {^1\xi_1} + \sum_{{\xi_i}> {^1\xi_1}}c_i {\xi_i} \right).
\end{align*}
Since $^1\xi_1$ is the smallest basis, the action will be 
\begin{align*}
E^{a(\xi_i)}\cdot v & =\lambda_{\beta}\beta +  \sum_{\Gamma_{i,j}>\beta}\lambda_i\Gamma_{i,j}\\
&=\lambda_{\beta}\beta.
\end{align*}
Therefore,
$$E^{a(\xi_i)}\cdot v=\lambda_{\beta}\beta. $$
\end{proof}
This implies 
\begin{cor}
If $S \subset R$ is a non-zero submodule, then $\beta \in S$.
\end{cor}
We proved from Proposition \ref{5.1.2} that there is a highest weight vector $\beta \in R$. So if $S \subset R$ is a non-zero submodule, then $S=R$. This implies that $\beta \in S$.
\begin{lem}\label{5.1.6}
Let $$ B = \left\{ \underbar{a}(\xi) \mid \xi \text{ is a basis element} \right\}.$$
 Then $$\left\{\underbar{F}^{a(\xi_i)}\cdot \beta \mid a(\xi)\in B\right\}$$ 
 is a basis.
\end{lem}
\begin{proof}

Now, we want to show that $R$ is generated by $\beta$.

Form the ordering in Proposition \ref{5.1.2}, we see that at least  
$$F_{i+1,i}\cdot ^1\xi_1< F_{i+1,i}\cdot \xi_2.$$
Also, for $j\neq i$
$$F_{j+1,j} \cdot ^1\xi_1< F_{j+1,j}\cdot \xi_2.$$
Therefore, we can write $$F^{\underbar{a}(\xi_i)}\cdot \beta=\chi_i + \sum_{\gamma_{i,j}>\chi_i}\gamma_{i,j}$$
 where $\xi_1 = \chi_i$ and $\xi_2 = \sum_{\gamma_{i,j}>\chi_i}\gamma_{i,j}$.
 
Suppose 
$$\sum_{i=1}^Nc_i F^{\underbar{a}(\xi_i)}\cdot \beta=0$$
and all $c_i\neq 0$. Then $$\sum_{i=1}^Nc_i\left(\chi_i + \sum_{\gamma_{i,j}>\chi_i}\gamma_{i,j}\right)=0.$$ We can fix $\chi_i$ such that $\chi_1 < \chi_2 < \cdots < \chi_N$. So $$\sum_{i=1}^Nc_i\left(\chi_i + \sum_{\gamma_{i,j}>\chi_i}\gamma_{i,j}\right)= c_1\chi_1 + \sum_{i=2}^Nc_i\left(\chi_i + \sum_{\gamma_{i,j}>\chi_i}\gamma_{i,j}\right)+\gamma_1=0.$$ We know $\chi_1$ is smaller than everything and $\{\chi\}$ are linearly independent, then $c_1=0$. Then 
$$\left\{F^{\underbar{a}(\xi_i)}\cdot \beta \mid\underbar{a}(\xi)\in B\right\}$$ 
is linearly independent.

We are given that ${\xi_i}$ is a basis of $R$ implying $\xi$ is a basis element. The cardinality of $\xi$ is $D$ (that is $\dim R$, in other words the number of basis vectors one can make from a given partition). Since $\xi$ has $N$ linearly independent elements, then
\begin{align*}
\dim \left\langle F^{\underbar{a}(\xi_i)}\cdot \beta \right\rangle = \dim R = D.
\end{align*}
So 
\begin{align*}
\left\{F^{\underbar{a}(\xi_i)}\cdot \beta \mid\underbar{a}(\xi)\in B\right\}
\end{align*}
 spans and is a basis in $R$. Since 
 \begin{align*}
 \left\{F^{\underbar{a}(\xi_i)}\cdot \beta \mid\underbar{a}(\xi)\in B\right\}
\end{align*}   
spans $R$ and all its elements are linearly independent, then it is all of $R$. Therefore, the weight vector $\beta$ generates all of $R$.  
\end{proof}
From the above proofs, we can make out that if $\beta$ is a highest weight vector, $S$ a submodule of $R$ (i.e $\beta \in S$ and $S$ is all of $R$) then $\beta$ generates all of $R$. Therefore, there is no invariant subspace of $R$. 
\begin{cor}
$R$ is generated by $\beta$, and moreover if $S \subset R$ is a non-zero submodule, then $S = R$.
\end{cor}
These completes the proof for Theorem \ref{5.1.1}. So, the representation space $R$ is a simple $sl_n-$module. Already, a monomial basis is constructed in Lemma \ref{5.1.6} as part of the proof for the theorem above and we will treat in details the procedure of forming these monomial basis. 
\section{Monomial basis}\label{5.2} We have now constructed a monomial basis in Lemma \ref{5.1.6}. But we have made a choice, by fixing the elements and their order on how to be applied to the highest weight vector.
 
\begin{exa}\label{4.2.1} Let $m_1=2,m_2=1,m_3=0$. We apply the sequences 
$$E_{1,2}^{a_1}E_{2,3}^{a_2}E_{1,2}^{a_3}$$
on $\xi$:

\begin{eqnarray*}
 E_{1,2}^{0}E_{2,3}^{0}E_{1,2}^{0}\cdot \left(\begin{matrix}
2 && 1 \\
 & 2 &
 \end{matrix}\right) &\text{ has }& (0,0,0) \hspace{.5cm} \Rightarrow F_{1,2}^{0}F_{2,3}^{0}F_{1,2}^{0} \text{ is a monomial,}\\\\
 E_{1,2}^{0}E_{2,3}^{1}E_{1,2}^{0}\cdot \left(\begin{matrix}
2 && 0 \\
 & 2 &
 \end{matrix}\right) &\text{ has }& (0,1,0)\hspace{.5cm} \Rightarrow F_{1,2}^{0}F_{2,3}^{1}F_{1,2}^{0} \text{ is a monomial,}\\\\
E_{1,2}^{0}E_{2,3}^{0}E_{1,2}^{1}\cdot \left(\begin{matrix}
2 && 1 \\
 & 1 &
 \end{matrix}\right) &\text{ has }& (0,0,1) \hspace{.5cm}\Rightarrow F_{1,2}^{1}F_{2,3}^{0}F_{1,2}^{0} \text{ is a monomial,}\\\\
E_{1,2}^{0}E_{2,3}^{1}E_{1,2}^{1}\cdot \left(\begin{matrix}
2 && 0 \\
 & 1 &
 \end{matrix}\right) &\text{ has }& (0,1,1)\hspace{.5cm} \Rightarrow F_{1,2}^{1}F_{2,3}^{1}F_{1,2}^{0} \text{ is a monomial,}\\\\
E_{1,2}^{0}E_{2,3}^{1}E_{1,2}^{2}\cdot \left(\begin{matrix}
2 && 0 \\
 & 0 &
 \end{matrix}\right) &\text{ has }& (0,1,2)\hspace{.5cm} \Rightarrow F_{1,2}^{2}F_{2,3}^{1}F_{1,2}^{0} \text{ is a monomial,}\\\\
E_{1,2}^{1}E_{2,3}^{1}E_{1,2}^{0}\cdot \left(\begin{matrix}
1 && 1 \\
 & 1 &
 \end{matrix}\right) &\text{ has }& (1,1,0)\hspace{.5cm} \Rightarrow F_{1,2}^{0}F_{2,3}^{1}F_{1,2}^{1} \text{ is a monomial,}\\\\
E_{1,2}^{1}E_{2,3}^{2}E_{1,2}^{0}\cdot \left(\begin{matrix}
1 && 0 \\
 & 1 &
 \end{matrix}\right) &\text{ has }& (1,2,0) \hspace{.5cm}\Rightarrow F_{1,2}^{0}F_{2,3}^{2}F_{1,2}^{1} \text{ is a monomial,}\\\\
 E_{1,2}^{1}E_{2,3}^{2}E_{1,2}^{1}\cdot \left(\begin{matrix}
1 && 0 \\
 & 0 &
 \end{matrix}\right) &\text{ has }& (1,2,1) \hspace{.5cm} \Rightarrow F_{1,2}^{1}F_{2,3}^{2}F_{1,2}^{1} \text{ is a monomial.}\\\\
\end{eqnarray*}
So
\begin{align*}
&\left\{F_{1,2}^{0}F_{2,3}^{0}F_{1,2}^{0},F_{1,2}^{0}F_{2,3}^{1}F_{1,2}^{0}, F_{1,2}^{1}F_{2,3}^{0}F_{1,2}^{0}, F_{1,2}^{1}F_{2,3}^{1}F_{1,2}^{0}, F_{1,2}^{2}F_{2,3}^{1}F_{1,2}^{0}, F_{1,2}^{0}F_{2,3}^{1}F_{1,2}^{1},\right.\\
& \left. F_{1,2}^{0}F_{2,3}^{2}F_{1,2}^{1},  F_{1,2}^{1}F_{2,3}^{2}F_{1,2}^{1} \right\}
\end{align*}  
is a monomial basis.

Also, the action of  
$$  E_{2,3}^{b_1}E_{1,2}^{b_2}E_{2,3}^{b_3}$$ 
on $\xi$ is demonstrated below.

\begin{eqnarray*}
E_{2,3}^{0}E_{1,2}^{0}E_{2,3}^{0}\cdot \left(\begin{matrix}
2 && 1 \\
 & 2 &
 \end{matrix}\right) &\text{ has }& (0,0,0)\hspace{.5cm} \Rightarrow F_{2,3}^{0}F_{1,2}^{0}F_{2,3}^{0} \text{ is a monomial,}\\\\
E_{2,3}^{0}E_{1,2}^{0}E_{2,3}^{1}\cdot \left(\begin{matrix}
2 && 0 \\
 & 2 &
 \end{matrix}\right) &\text{ has }& (0,0,1)\hspace{.5cm} \Rightarrow F_{2,3}^{1}F_{1,2}^{0}F_{2,3}^{0} \text{ is a monomial,}\\\\
E_{2,3}^{0}E_{1,2}^{1}E_{2,3}^{0}\cdot \left(\begin{matrix}
2 && 1 \\
 & 1 &
 \end{matrix}\right) &\text{ has }& (0,1,0)\hspace{.5cm} \Rightarrow F_{2,3}^{0}F_{1,2}^{1}F_{2,3}^{0} \text{ is a monomial,}\\\\
 E_{2,3}^{0}E_{1,2}^{1}E_{2,3}^{1}\cdot \left(\begin{matrix}
2 && 0 \\
 & 1 &
 \end{matrix}\right) &\text{ has }& (0,1,1)\hspace{.5cm} \Rightarrow F_{2,3}^{1}F_{1,2}^{1}F_{2,3}^{0} \text{ is a monomial,}\\\\
E_{2,3}^{1}E_{1,2}^{2}E_{2,3}^{0}\cdot \left(\begin{matrix}
2 && 0 \\
 & 0 &
 \end{matrix}\right) &\text{ has }& (1,2,0)\hspace{.5cm} \Rightarrow F_{2,3}^{0}F_{1,2}^{2}F_{2,3}^{1} \text{ is a monomial,}\\\\
E_{2,3}^{0}E_{1,2}^{1}E_{2,3}^{1}\cdot \left(\begin{matrix}
1 && 1 \\
 & 1 &
 \end{matrix}\right) &\text{ has }& (0,1,1)\hspace{.5cm} \Rightarrow F_{2,3}^{1}F_{1,2}^{1}F_{2,3}^{0} \text{ is a monomial,}\\\\
E_{2,3}^{0}E_{1,2}^{1}E_{2,3}^{2}\cdot \left(\begin{matrix}
1 && 0 \\
 & 1 & 
 \end{matrix}\right) &\text{ has }& (0,1,2)\hspace{.5cm} \Rightarrow F_{2,3}^{2}F_{1,2}^{1}F_{2,3}^{0} \text{ is a monomial,}\\\\
E_{2,3}^{1}E_{1,2}^{2}E_{2,3}^{1}\cdot \left(\begin{matrix}
1 && 0 \\
 & 0 &
 \end{matrix}\right) &\text{ has }& (1,2,1)\hspace{.5cm} \Rightarrow F_{2,3}^{1}F_{1,2}^{2}F_{2,3}^{1} \text{ is a monomial.}\\
\end{eqnarray*}
So 
\begin{align*}
&\left\{ F_{2,3}^{0}F_{1,2}^{0}F_{2,3}^{0}, F_{2,3}^{1}F_{1,2}^{0}F_{2,3}^{0}, F_{2,3}^{0}F_{1,2}^{1}F_{2,3}^{0}, F_{2,3}^{1}F_{1,2}^{1}F_{2,3}^{0}, F_{2,3}^{0}F_{1,2}^{2}F_{2,3}^{1},
F_{2,3}^{2}F_{1,2}^{1}F_{2,3}^{0}, F_{2,3}^{1}F_{1,2}^{2}F_{2,3}^{1}  \right\}
\end{align*}
is a monomial basis.
\end{exa} 
We see that different sequences make different monomial basis. Since the partition in Example \ref{4.2.1} is symmetric, we do another example.

\begin{exa} Let $m_1=3,m_2=2,m_3=0$. We apply the sequences 
$$E_{1,2}^{a_1}E_{2,3}^{a_2}E_{1,2}^{a_3}$$
on $\xi$ to find a monomial.

\begin{eqnarray*}
 E_{1,2}^{0}E_{2,3}^{0}E_{1,2}^{0}\cdot \left(\begin{matrix}
3 && 2 \\
 & 3 &
 \end{matrix}\right) &\text{ has }& (0,0,0)\hspace{.5cm} \Rightarrow F_{1,2}^{0}F_{2,3}^{0}F_{1,2}^{0} \text{ is a monomial,}\\\\
 E_{1,2}^{0}E_{2,3}^{1}E_{1,2}^{0}\cdot \left(\begin{matrix}
3 && 1 \\
 & 3 &
 \end{matrix}\right) &\text{ has }& (0,1,0) \hspace{.5cm}\Rightarrow F_{1,2}^{0}F_{2,3}^{1}F_{1,2}^{0} \text{ is a monomial,}\\\\
E_{1,2}^{0}E_{2,3}^{2}E_{1,2}^{0}\cdot \left(\begin{matrix}
3 && 0 \\
 & 3 &
 \end{matrix}\right) &\text{ has }& (0,2,0)\hspace{.5cm} \Rightarrow F_{1,2}^{0}F_{2,3}^{2}F_{1,2}^{0} \text{ is a monomial,}\\\\
  E_{1,2}^{0}E_{2,3}^{0}E_{1,2}^{1}\cdot \left(\begin{matrix}
3 && 2 \\
 & 2 &
 \end{matrix}\right) &\text{ has }& (0,0,1)\hspace{.5cm} \Rightarrow F_{1,2}^{1}F_{2,3}^{0}F_{1,2}^{0} \text{ is a monomial,}\\\\
E_{1,2}^{0}E_{2,3}^{1}E_{1,2}^{1}\cdot \left(\begin{matrix}
3 && 1 \\
 & 2 &
 \end{matrix}\right) &\text{ has }& (0,1,1)\hspace{.5cm} \Rightarrow F_{1,2}^{1}F_{2,3}^{1}F_{1,2}^{0} \text{ is a monomial,}\\\\
E_{1,2}^{0}E_{2,3}^{2}E_{1,2}^{1}\cdot \left(\begin{matrix}
3 && 0 \\
 & 2 &
 \end{matrix}\right) &\text{ has }& (0,2,1)\hspace{.5cm} \Rightarrow F_{1,2}^{1}F_{2,3}^{2}F_{1,2}^{0} \text{ is a monomial,}\\\\
E_{1,2}^{0}E_{2,3}^{1}E_{1,2}^{2}\cdot \left(\begin{matrix}
3 && 1 \\
 & 1 &
 \end{matrix}\right) &\text{ has }& (0,1,2)\hspace{.5cm} \Rightarrow F_{1,2}^{2}F_{2,3}^{1}F_{1,2}^{0} \text{ is a monomial,}\\\\
 E_{1,2}^{0}E_{2,3}^{2}E_{1,2}^{2}\cdot \left(\begin{matrix}
3 && 0 \\
 & 1 &
 \end{matrix}\right) &\text{ has }& (0,2,2)\hspace{.5cm} \Rightarrow F_{1,2}^{2}F_{2,3}^{2}F_{1,2}^{0} \text{ is a monomial.}\\\\
 E_{1,2}^{0}E_{2,3}^{2}E_{1,2}^{3}\cdot \left(\begin{matrix}
3 && 0 \\
 & 0 &
 \end{matrix}\right) &\text{ has }& (0,2,3)\hspace{.5cm} \Rightarrow F_{1,2}^{3}F_{2,3}^{2}F_{1,2}^{0} \text{ is a monomial,}\\\\
 E_{1,2}^{1}E_{2,3}^{1}E_{1,2}^{0}\cdot \left(\begin{matrix}
2 && 2 \\
 & 2 &
 \end{matrix}\right) &\text{ has }& (1,1,0)\hspace{.5cm} \Rightarrow F_{1,2}^{0}F_{2,3}^{1}F_{1,2}^{1} \text{ is a monomial,}\\\\
E_{1,2}^{1}E_{2,3}^{2}E_{1,2}^{0}\cdot \left(\begin{matrix}
2 && 1 \\
 & 2 &
 \end{matrix}\right) &\text{ has }& (1,2,0)\hspace{.5cm} \Rightarrow F_{1,2}^{0}F_{2,3}^{2}F_{1,2}^{1} \text{ is a monomial,}\\\\
E_{1,2}^{1}E_{2,3}^{3}E_{1,2}^{0}\cdot \left(\begin{matrix}
2 && 0 \\
 & 2 &
 \end{matrix}\right) &\text{ has }& (1,3,0)\hspace{.5cm} \Rightarrow F_{1,2}^{0}F_{2,3}^{3}F_{1,2}^{1} \text{ is a monomial,}\\\\
E_{1,2}^{1}E_{2,3}^{2}E_{1,2}^{1}\cdot \left(\begin{matrix}
2 && 1 \\
 & 1 &
 \end{matrix}\right) &\text{ has }& (1,2,1)\hspace{.5cm} \Rightarrow F_{1,2}^{1}F_{2,3}^{2}F_{1,2}^{1} \text{ is a monomial,}\\\\
E_{1,2}^{1}E_{2,3}^{3}E_{1,2}^{1}\cdot \left(\begin{matrix}
2 && 0 \\
 & 1 &
 \end{matrix}\right) &\text{ has }& (1,3,1) \hspace{.5cm} \Rightarrow F_{1,2}^{1}F_{2,3}^{3}F_{1,2}^{1} \text{ is a monomial,}\\\\
E_{1,2}^{1}E_{2,3}^{3}E_{1,2}^{2}\cdot \left(\begin{matrix}
2 && 0 \\
 & 0 &
 \end{matrix}\right) &\text{ has }& (1,3,2)\hspace{.5cm} \Rightarrow F_{1,2}^{2}F_{2,3}^{3}F_{1,2}^{1} \text{ is a monomial,}
\end{eqnarray*}
So 
\begin{align*}
& \left\{ F_{1,2}^{0}F_{2,3}^{0}F_{1,2}^{0}, F_{1,2}^{0}F_{2,3}^{1}F_{1,2}^{0}, F_{1,2}^{0}F_{2,3}^{2}F_{1,2}^{0}, F_{1,2}^{1}F_{2,3}^{0}F_{1,2}^{0}, F_{1,2}^{1}F_{2,3}^{1}F_{1,2}^{0}, F_{1,2}^{1}F_{2,3}^{2}F_{1,2}^{0}, F_{1,2}^{2}F_{2,3}^{1}F_{1,2}^{0}, \right.\\ & \left. F_{1,2}^{2}F_{2,3}^{2}F_{1,2}^{0}, F_{1,2}^{3}F_{2,3}^{2}F_{1,2}^{0}, F_{1,2}^{0}F_{2,3}^{1}F_{1,2}^{1} , F_{1,2}^{0}F_{2,3}^{2}F_{1,2}^{1}, F_{1,2}^{0}F_{2,3}^{3}F_{1,2}^{1}, F_{1,2}^{1}F_{2,3}^{2}F_{1,2}^{1}, F_{1,2}^{1}F_{2,3}^{3}F_{1,2}^{1},\right. \\
& \left. F_{1,2}^{2}F_{2,3}^{3}F_{1,2}^{1} \right\}
\end{align*}
is a monomial basis.

Now, the action of 
$$  E_{2,3}^{b_1}E_{1,2}^{b_2}E_{2,3}^{b_3}$$ 
on $\xi$:

\begin{eqnarray*}
E_{2,3}^{0}E_{1,2}^{0}E_{2,3}^{0}\cdot \left(\begin{matrix}
3 && 2 \\
 & 3 &
 \end{matrix}\right) &\text{ has }& (0,0,0)\hspace{.5cm} \Rightarrow F_{2,3}^{0}F_{1,2}^{0}F_{2,3}^{0} \text{ is a monomial,}\\\\
E_{2,3}^{0}E_{1,2}^{0}E_{2,3}^{1}\cdot \left(\begin{matrix}
3 && 1 \\
 & 3 &
 \end{matrix}\right) &\text{ has }& (0,0,1)\hspace{.5cm} \Rightarrow F_{2,3}^{1}F_{1,2}^{0}F_{2,3}^{0} \text{ is a monomial,}\\\\
E_{2,3}^{0}E_{1,2}^{0}E_{2,3}^{2}\cdot \left(\begin{matrix}
3 && 0 \\
 & 3 &
 \end{matrix}\right) &\text{ has }& (0,0,2)\hspace{.5cm} \Rightarrow F_{2,3}^{2}F_{1,2}^{0}F_{2,3}^{0} \text{ is a monomial,}\\\\
E_{2,3}^{0}E_{1,2}^{1}E_{2,3}^{0}\cdot \left(\begin{matrix}
3 && 2 \\
 & 2 &
 \end{matrix}\right) &\text{ has }& (0,1,0)\hspace{.5cm} \Rightarrow F_{2,3}^{0}F_{1,2}^{1}F_{2,3}^{0} \text{ is a monomial,}\\\\
E_{2,3}^{0}E_{1,2}^{1}E_{2,3}^{1}\cdot \left(\begin{matrix}
3 && 1 \\
 & 2 &
 \end{matrix}\right) &\text{ has }& (0,1,1)\hspace{.5cm} \Rightarrow F_{2,3}^{1}F_{1,2}^{1}F_{2,3}^{0} \text{ is a monomial,}\\\\
 E_{2,3}^{0}E_{1,2}^{1}E_{2,3}^{2}\cdot \left(\begin{matrix}
3 && 0 \\
 & 2 &
 \end{matrix}\right) &\text{ has }& (0,1,2)\hspace{.5cm} \Rightarrow F_{2,3}^{2}F_{1,2}^{1}F_{2,3}^{0} \text{ is a monomial,}\\\\
E_{2,3}^{1}E_{1,2}^{2}E_{2,3}^{0}\cdot \left(\begin{matrix}
3 && 1 \\
 & 1 & 
 \end{matrix}\right) &\text{ has }& (1,2,0)\hspace{.5cm} \Rightarrow F_{2,3}^{0}F_{1,2}^{2}F_{2,3}^{1} \text{ is a monomial,}\\\\
E_{2,3}^{1}E_{1,2}^{2}E_{2,3}^{1}\cdot \left(\begin{matrix}
3 && 0 \\
 & 1 &
 \end{matrix}\right) &\text{ has }& (1,2,1)\hspace{.5cm} \Rightarrow F_{2,3}^{1}F_{1,2}^{2}F_{2,3}^{1} \text{ is a monomial,}\\\\
E_{2,3}^{2}E_{1,2}^{3}E_{2,3}^{0}\cdot \left(\begin{matrix}
3 && 0 \\
 & 0 &
 \end{matrix}\right) &\text{ has }& (2,3,0)\hspace{.5cm} \Rightarrow F_{2,3}^{0}F_{1,2}^{3}F_{2,3}^{2} \text{ is a monomial,}\\\\
E_{2,3}^{0}E_{1,2}^{1}E_{2,3}^{1}\cdot \left(\begin{matrix}
2 && 2 \\
 & 2 &
 \end{matrix}\right) &\text{ has }& (0,1,1)\hspace{.5cm} \Rightarrow F_{2,3}^{1}F_{1,2}^{1}F_{2,3}^{0} \text{ is a monomial,}\\\\
 E_{2,3}^{0}E_{1,2}^{1}E_{2,3}^{2}\cdot \left(\begin{matrix}
2 && 1 \\
 & 2 &
 \end{matrix}\right) &\text{ has }& (0,1,2)\hspace{.5cm} \Rightarrow F_{2,3}^{2}F_{1,2}^{1}F_{2,3}^{0} \text{ is a monomial,}\\\\
 \end{eqnarray*}
 \begin{eqnarray*}
E_{2,3}^{0}E_{1,2}^{1}E_{2,3}^{3}\cdot \left(\begin{matrix}
2 && 0 \\
 & 2 &
 \end{matrix}\right) &\text{ has }& (0,1,3) \hspace{.5cm}\Rightarrow F_{2,3}^{3}F_{1,2}^{1}F_{2,3}^{0} \text{ is a monomial,}\\\\
E_{2,3}^{1}E_{1,2}^{2}E_{2,3}^{1}\cdot \left(\begin{matrix}
2 && 1 \\
 & 1 &
 \end{matrix}\right) &\text{ has }& (1,2,1)\hspace{.5cm} \Rightarrow F_{2,3}^{1}F_{1,2}^{2}F_{2,3}^{1} \text{ is a monomial,}\\\\
E_{2,3}^{1}E_{1,2}^{2}E_{2,3}^{2}\cdot \left(\begin{matrix}
2 && 0 \\
 & 1 & 
 \end{matrix}\right) &\text{ has }& (1,2,2) \hspace{.5cm}\Rightarrow F_{2,3}^{2}F_{1,2}^{2}F_{2,3}^{1} \text{ is a monomial,}\\\\
E_{2,3}^{2}E_{1,2}^{3}E_{2,3}^{1}\cdot \left(\begin{matrix}
2 && 0 \\
 & 0 &
 \end{matrix}\right) &\text{ has }& (2,3,1)\hspace{.5cm} \Rightarrow F_{2,3}^{1}F_{1,2}^{3}F_{2,3}^{2} \text{ is a monomial.}
\end{eqnarray*}
So 
\begin{align*}
& \left\{  F_{2,3}^{0}F_{1,2}^{0}F_{2,3}^{0}, F_{2,3}^{1}F_{1,2}^{0}F_{2,3}^{0}, F_{2,3}^{2}F_{1,2}^{0}F_{2,3}^{0}, F_{2,3}^{0}F_{1,2}^{1}F_{2,3}^{0}, F_{2,3}^{1}F_{1,2}^{1}F_{2,3}^{0}, 
 F_{2,3}^{2}F_{1,2}^{1}F_{2,3}^{0},  \right. \\ 
 & \left.F_{2,3}^{0}F_{1,2}^{2}F_{2,3}^{1},
 F_{2,3}^{1}F_{1,2}^{2}F_{2,3}^{1}, F_{2,3}^{0}F_{1,2}^{3}F_{2,3}^{2}, F_{2,3}^{3}F_{1,2}^{1}F_{2,3}^{0},  F_{2,3}^{2}F_{1,2}^{2}F_{2,3}^{1},
F_{2,3}^{1}F_{1,2}^{3}F_{2,3}^{2}  \right\}
\end{align*}
 is a monomial basis.

So we see different sequences make different monomial basis.There are infinitely many monomials.
\end{exa}

 % You may have more chapters. (Use e.g. git add FILE)
\chapter{Conclusion}\label{6}
In this piece of work, our representation is actually 
\begin{align*}
\rho : sl_n & \longrightarrow \text{End}(R) \\
x & \longmapsto \rho(x).
\end{align*}
The map $\rho$ is linear and also the identity. We saw that the actions of upper triangular matrices and lower triangular matrices on a basis vector $\xi$ resulted in a new basis vector while the diagonal matrices act by a scaler. Suppose $v\in R$ and 
\begin{align*}
v=\lambda_1\xi_1+\cdots + \lambda_n\xi_n,
\end{align*}
where $\xi_1, \cdots , \xi_n$ are basis vectors and $\lambda_1, \cdots, \lambda_n$ are non-zero coefficients. Let $\lambda_1\neq 0$ and $\lambda_2=\cdots =\lambda_n$. Then
\begin{align*}
F\cdot v=\lambda_1F\cdot\xi_1
\end{align*}
and 
\begin{align*}
E\cdot(F\cdot v)=\lambda_1E\cdot(F\cdot\xi_1)
\end{align*}
are both well defined operations in our representation. Now, let $\lambda_1=\lambda_3=\cdots =\lambda_n$ and $\lambda_2\neq 0$. Then
\begin{align*}
E\cdot v=\lambda_2E\cdot\xi_2
\end{align*}
and 
\begin{align*}
F\cdot(E\cdot v)=\lambda_2F\cdot(E\cdot\xi_2)
\end{align*}
again are both well defined operations in our representation. The diagonal matrices act by a scaler; that is 
\begin{align*}
H\cdot \xi = \kappa \xi
\end{align*}
 and so the basis vector acted on remains the same. In all the actions above, the results are all accounted for in formulas of Equations \eqref{3.0.3}, \eqref{3.0.5} and \eqref{3.0.6}. If $\xi_1, \cdots , \xi_n\in S \Rightarrow S=R$. So, $\rho$ has no invariant subspace. Therefore, $\rho$ is an irreducible representation of special linear algebra, $sl_n$.

The $sl_n-$modules are very interesting as we study their actions on basis vectors. A fundamental component of this project is to show that $sl_n-$modules are actually simple and can form monomial basis. For any partition, we can construct all possible basis vectors and modules as explicitly explained in  Chapter \ref{1}. Because these bases vectors are complex combinatoric structures, we apply total ordering on them to identify the smallest basis vector. A sequence of upper triangular matrices that acts maximally on the smallest bases vector will eventually act on a set of bases vectors resulting in a total annihilation of all bases vectors but raising the smallest basis vector maximally. This is a very effective and efficient method of analysis. So, a maximally raised weight vector is then a highest weight vector. We proved that $\exists$ a highest weight vector and it has weight 
\begin{align*}
\omega_i = c_1\varepsilon_1 + \cdots + c_n\varepsilon_n.
\end{align*}  
We also proved in that every basis vector has a sequence of upper triangular matrices that acts on it maximally to yield a highest weight vector. Furthermore, suppose there were many weight vectors. There $\exists$ a sequence of upper triangular matrices whose action yields a highest weight vector. This is efficiently done by employing total order. We proved the existence of monomial basis and gave a construction. Each of these facts contributes in showing that $sl_n-$module is simple and has monomial basis.
% This is where we stop counting pages.
% Acknowledgements and References are not counted.
%-----------------------------------------------------------------------------
% Note the errata page is not for now, it is for use during the examination.
% Not that you're going to have any errata.
%-----------------------------------------------------------------------------
% THE BIBLIOGRAPHY 
% Bibliography styles define how the bibliography is 
% listed and formatted. This is part of the AIMS house
% style and is only changed under exceptional circumstances
\renewcommand{\bibname}{References}
\nocite{*}
\bibliographystyle{unsrt}
\bibliography{references}
\addcontentsline{toc}{chapter}{References}
%-----------------------------------------------------------------------------
\end{document}